\documentclass[12pt]{amsart}
\usepackage{amsmath,amsthm,amsfonts,amscd,amssymb,eucal,latexsym,mathrsfs, appendix, yhmath, setspace}
\usepackage{color,colortbl}
\usepackage[numbers,sort&compress]{natbib}
\usepackage[all,cmtip]{xy}
\usepackage{enumerate}

\newtheorem{theorem}{Theorem}[section]
\newtheorem{corollary}[theorem]{Corollary}
\newtheorem{lemma}[theorem]{Lemma}
\newtheorem{proposition}[theorem]{Proposition}

\theoremstyle{definition}
\newtheorem{definition}[theorem]{Definition}
\newtheorem{remark}[theorem]{Remark}

\newtheorem{example}[theorem]{Example}
\newtheorem{question}[theorem]{Question}

\DeclareMathOperator{\Ima}{im \ }

\newcommand{\gen}{{\rm gen}}

\newcommand{\rk}{{\rm rk}}
\newcommand{\mrk}{{\rm mrk}}

\newcommand{\mdim}{{\rm mdim}}
\newcommand{\ml}{{\rm m\ell}}
\newcommand{\mL}{{\rm mL}}

\newcommand{\rL}{{\rm L}}

\newcommand{\tr}{{\rm tr}}

  \newcommand{\cC}{{\mathcal C}}
  
  \newcommand{\cF}{{\mathcal F}}

 \newcommand{\cM}{{\mathcal M}}
 \newcommand{\cN}{{\mathcal N}}
 \newcommand{\cP}{{\mathcal P}}

 \newcommand{\bN}{{\mathbb N}}
 
 \newcommand{\bQ}{{\mathbb Q}}
 \newcommand{\bR}{{\mathbb R}}
 
 \newcommand{\bZ}{{\mathbb Z}}
  \newcommand{\RG}{{R \Gamma}}
 \newcommand{\ZG}{{\mathbb Z \Gamma}}
\newcommand{\LG}{{\mathcal L \Gamma }}

\newcommand{\sN}{{\mathscr N}}

\allowdisplaybreaks

\begin{document}

\title{Mean weak length}

\author{Zihan Bai}

\address{\hskip-\parindent
Z.B., School of Mathematical Sciences, Soochow University,
Suzhou 215006, China.}
\email{20254007003@stu.suda.edu.cn}

\author{Bingbing Liang}
\address{\hskip-\parindent
B.L., School of Mathematical Sciences, Soochow University,
Suzhou 215006, China.}
\email{bbliang@suda.edu.cn}


\begin{abstract}
We introduce a weak version of the classical length function, termed the weak length function, defined on subsets of $R$-modules relative to the ambient $R$-modules over a unital ring $R$. We further consider the concept of mean weak length for each $R\Gamma$-module relative to a bigger $R\Gamma$-module  associated with an amenable group $\Gamma$. Under an appropriate upgrading condition together with certain mild assumptions, we establish that the mean weak length satisfies an  addition formula with respect to short exact sequences. 

This result has three applications. First, we provide a purely algebraic proof of the additivity of algebraic entropy, which is a property originally established via topological entropy methods. Second, within our unified framework, we give an alternative and conceptual proof of the additivity of mean length, previously obtained in \cite{LL18} and \cite{V19} using different approaches. Third, we recover a particular case of amenable extension for Sylvester rank functions in the spirit of \cite{JL21}.

\end{abstract}

\maketitle

\section{Introduction}

The validity of additivity is a central problem for invariants on pure algebraic objects. In the category of $R$-modules for a given unital ring $R$, a numerical invariant ${\rm i}$ is said to have the property of {\it additivity}, or the so-called {\it addition formula}, if for any left  $R$-modules $\cM_1 \subseteq \cM_2$, one has
$${\rm i}(\cM_2)={\rm i}(\cM_1)+{\rm i}(\cM_2/\cM_1).$$
The notion of mean length, termed in \cite{LL18, LL19}, is a class of numerical invariants for modules over the group ring $R\Gamma$ of a discrete group $\Gamma$ with coefficients in $R$. Its definition is based on a length function $\rL$ on $R$-modules introduced by Northcott-Reufel \cite{NR65}.   Salce-Zanardo first initiated an axiomatic approach for $\Gamma=\bZ$ \cite{SZ09} and the addition formula was established by Scale-V\'{a}mos-Virili in the case that $\rL$ is discrete \cite{SVV13}. Salce-Virili considered the case $\Gamma=\bZ^d$ and $\rL$ is nondiscrete \cite{SV16}. Elek considered the case that $R$ is any field and $\Gamma$ is any amenable group \cite{E03}. Li-Liang dealt with the case that $\rL(R)< \infty$ for amenable groups in \cite{LL18} and sofic groups in \cite{LL19}. Virili established the additivity for a general length $\rL$ and any amenable group $\Gamma$  in \cite{V19}. Note that the validity of additivity for mean length is restricted to the locally finite modules in the sense that all finitely generated submodules have finite length values.

On the other hand, algebraic entropy for $\ZG$-modules defined by Peters attracted great attention due to its direct connection with the topological entropy of Pontryagin dual actions \cite{P79}, which is the so-called {\it Bridge theorem}. Motivated by the connection between Mahler measure and topological entropy established by Yuzvinskii\cite{Y67}, one can even use the algebraic entropy to give a direct characterization of Lehmer's problem \cite{LSW90, DG16}.  Unfortunately, the definition of algebraic entropy is not based on a length function (though torsion modules are locally finite and can be dealt within the tool of mean length) and its additivity remains inaccessible to standard mean length approaches.  Taking advantage of Bridge theorem, for any amenable group $\Gamma$, Li confirmed the validity of additivity of algebraic entropy by establishing the additivity of topological entropy for a class of dynamical systems including algebraic actions \cite{L12}.  This motivates the following questions:

\begin{question}
    Is there a pure algebraic proof of additivity for algebraic entropy? Moreover, is there a unified approach of additivity for algebraic entropy and mean length?
\end{question}

 There has been much study on the additivity for algebraic entropy \cite{W74}. A complete solution for the case $\Gamma=\bZ$ is given by Dikranjan-Giordano Bruno \cite{DG16}.  Recently the additivity is established for actions of monotileable amenable groups by Dikranjan-Fornasiero-Giordano Bruno-Salizzoni \cite{DFGS23}. The crucial ingredient in  \cite{DFGS23} is to introduce a bivariant function on subsets of abelian groups as a main tool to estimate the size of a module in terms of submodules.

A complementary perspective arises from the systematic study of Sylvester module rank functions \cite{L21, JL21}. Although a Sylvester module rank function 
$\dim(\cdot)$ needs not satisfy the classical additivity property, it admits a unique bivariant extension $\dim(\cdot|\cdot)$ such that $\dim(\cM|\cM)=\dim(\cM)$  for every $R$-module $\cM$ (see \cite[Theorem 3.3]{L21}). With this bivariant notion, an alternative addition formula holds: for any $R$-modules $\cM_1 \subseteq \cM_2$,
 $$\dim(\cM_2)=\dim(\cM_1|\cM_2)+\dim(\cM_2/\cM_1).$$
 Building on this alternative addition formula, Jiang-Li vastly extended Sylvester module rank functions from $R$-modules to amenable $S$-modules, encompassing the amenable algebras over a field (in the sense of Gromov and Elek) and the group ring $R\Gamma$ for an amenable group $\Gamma$. All these results rely heavily on the assumption $\dim(R)<\infty$, which naturally raises the question: what can be said when $\dim(R)=\infty$?

In this paper, we establish a more flexible framework to address the issues mentioned above. To unify the distinct approaches of mean length and algebraic entropy, our idea is to first relax the defining requirements of a length function, thereby introducing a weaker variant that encompasses both the classical length function and the logarithmic function arising in the definition of algebraic entropy. We call such a function a {\it weak length function}, defined on the subsets of $R$-modules relative to the ambient modules. Based on this notion, we can define an invariant on pairs of $\RG$-modules $\cM_1 \subseteq \cM_2$, called {\it mean weak length}, to generalize both mean length and algebraic entropy.  To overcome the core difficulty of upper bound estimation in the addition formula, motivated from the bivariant functions in \cite{DFGS23}, we also introduce a bivariant version of weak length functions. It turns out that as a weak length function can be upgraded to a ``good" version, the desired upper bound estimate can be successfully achieved. Our main result is the following.

\begin{theorem} \label{main theorem}
Let $R$ be a unital ring and $\Gamma$ a discrete amenable group. Let $\ell$ be an  $R$-weak length function with the strong quotient property. If $\ell$ can be upgraded to a proper bivariant $R$-weak length function, then for any locally $\ell$-finite $\RG$-modules $\cM_1 \subseteq \cM_2 \subseteq \cM_3$, we have
$$\ml(\cM_2|\cM_3)=\ml(\cM_1|\cM_3)+\ml(\cM_2/\cM_1|\cM_3/\cM_1).$$
In particular, if $\ell$ is independent of ambient modules, the mean weak length $\ml$ on $\RG$-modules  satisfies the addition formula. That is,  
$$\ml(\cM_2)=\ml(\cM_1)+\ml(\cM_2/\cM_1).$$
\end{theorem}

After verifying the upgrading condition for algebraic entropy, mean length, and Sylvester module rank function (Propositions \ref{length upgrading} and \ref{log upgrading}), we obtain the following consequences:
\begin{corollary} \label{main cor}
\begin{enumerate}
    
    \item For any $\ZG$-modules $\cM_1 \subseteq \cM_2$, we have
$$h(\cM_2)=h(\cM_1)+h(\cM_2/\cM_1),$$
where $h(\cM)$ denotes the algebraic entropy of a $\ZG$-module $\cM$.
\item Suppose that $\rL$ is a length function on $R$-modules. Then for any locally $\rL$-finite $\RG$-modules $\cN_1 \subseteq \cN_2$, one has
$$\mL(\cN_2)=\mL(\cN_1)+\mL(\cN_2/\cN_1).$$
\item Suppose that $\dim$ is a bivariant Sylvester $R$-module rank function. Then for any $\RG$-modules $\cM_1 \subseteq \cM_2 \subseteq \cM_3$, we have
$$\mdim(\cM_2|\cM_3)=\mdim(\cM_1|\cM_3)+\mdim(\cM_2/\cM_1|\cM_3/\cM_1).$$
\end{enumerate}
\end{corollary}

Applying the Bridge theorem \cite[Theorem 13.41]{KLb}, we can in turn  provide an algebraic proof of the addition formula for algebraic actions.
\begin{corollary}
    For any compact Hausdorff abelian group $X$ carrying a continuous $\Gamma$-action by automorphisms and $Y \leq X$ as a $\Gamma$-invariant closed subgroup, one has
    $$h(X)=h(Y)+h(X/Y).$$
Here $X/Y$ is the quotient group carrying an induced $\Gamma$-action by quotients and $h(X)$ denotes the topological entropy of a continuous action $\Gamma \curvearrowright X$.
\end{corollary}

This paper is organized as follows. We recall some background knowledge in Section 2. The notions of weak length and its bivariant version are introduced in Section 3 and their basic properties are discussed. In addition, we verify that all these properties are satisfied for the $\log$ function, locally finite length functions, and Sylvester module rank functions.   In Section 4, we introduce the notion of mean weak length and discuss its basic properties. The addition formula is established in Section 5.

Throughout this paper, $\Gamma$ will be a discrete group with the identity element $e$. For any set $S$, we denote by $\cF(S)$ the set of all nonempty finite subsets of $S$. $R$ is always a unital ring. All modules are assumed to be left modules.

\noindent{\it Acknowledgements.}
We are grateful to Hanfeng Li for his helpful insight on the notion of mean length and  Proposition \ref{product infinity}, and to Yongle Jiang for his careful reading. The second author is supported by NSFC grant 12271387. This work was partially supported by the Simons Foundation grant (award no. SFI-MPS-T-Institutes-00010825) and from State Treasury funds as part of a task commissioned by the Minister of Science and Higher Education under the project ``Organization of the Simons Semesters at the Banach Center-New Energies in 2026-2028"(agreement no. MNiSW/2025/DAP/491).

\section{Preliminaries}

\subsection{Group rings and length functions}
Given a unital ring $R$, the {\it group ring of $\Gamma$ with coefficients in $R$}, denoted by $R \Gamma$, consists of all finitely supported maps from $\Gamma$ to $R$. For each element $f$ in $\RG$, conventionally, we shall write $f$ as $\sum_{s \in \Gamma} f_s s$, where $f_s \in R$ for all $s \in \Gamma$ and $f_s=0$ for all except finitely many $s \in \Gamma$. 
The ring structure on $R \Gamma$ is given by
$$ \sum_{s\in \Gamma}f_ss+\sum_{s\in \Gamma}g_ss=\sum_{s\in \Gamma}(f_s+g_s)s, \mbox{ and } \big(\sum_{s\in \Gamma}f_s s\big)\big(\sum_{t\in \Gamma}g_tt\big)=\sum_{s, t\in \Gamma}f_sg_t(st).$$
We similarly have the product if one of $f$ and $g$ sits in $R^\Gamma$.

Following  V\'{a}mos \cite{V68}, the length function is defined as follows:
\begin{definition} \label{D-length}
A map $\rL $ on $R$-modules with the value in $\bR_{\geq 0}\cup \{+\infty\}$ is a  {\bf length function} if the following conditions are satisfied:
\begin{enumerate}
\item $\rL(0)=0$;
\item {\bf additivity}: $\rL(\cM_2)=\rL(\cM_1)+\rL(\cM_3)$ holds for any short exact sequence $0\rightarrow \cM_1\rightarrow \cM_2\rightarrow \cM_3\rightarrow 0$ of $R$-modules;
\item {\bf upper continuity}: for any $R$-module $\cM$, one has $\rL(\cM)=\sup_\cN \rL(\cN)$ for $\cN$ ranging over all finitely generated $R$-submodules of $\cM$.
\end{enumerate}
We say an $R$-module $\cM$ is {\bf locally $\rL$-finite} if $\rL(\cN)$ is finite for any finitely generated $R$-submodule $\cN$ of $\cM$. $\rL$ is called {\bf locally finite} if it takes finite values on every finitely generated $R$-submodule.
\end{definition}

\begin{example} \label{wleg}
    \begin{enumerate}
        \item Using composition series, a length function $\rL$ on $R$-modules can be defined as
        $$\rL(\cM):=\sup\{n: 0=\cM_0\subsetneq \cM_1 \cdots \subsetneq \cM_n=\cM\},$$
        where each $\cM_i$ is a $R$-submodule of $\cM$. The fact that $\rL$ is a length function follows from the Jordan-H{o}lder theorem\cite[Example 2.2]{LL19};
        \item Let $R$ be a Dedekind domain, i.e. $R$ is a integrally-closed Noetherian domain whose all nonzero prime ideals are maximal ideals.  Recall that the classical primary decomposition theorem for a finitely generated $R$-module $\sN$ asserts that it is isomorphic to 
        $$R^{d-1}\oplus I \oplus \oplus_{j=1}^k (R/P_j^{e_j})$$
        for some integers $d, k \geq 0, e_j\geq 1$, some ideal $I$, and prime ideals $P_j$ of $R$ ($R^{-1}$ and $R^0$ are understood as zero). For every $R$-module $\cM$ define
        $$\nu(\cM):=\sum_{j=1}^k e_j$$
        if $\cM$ is a finitely generated torsion and $\nu(\cM)=+\infty$ otherwise.
      Then $\nu$ is a length function \cite[Example 1.3]{SZ09}. 
      
      Now assume that $R$ is a principal ideal domain as a special case. Then the isomorphism type above can be written as
      $$R^d \oplus \oplus_{j=1}^k (R/P_j^{e_j})$$
      To capture the information of the free part, we can define
      $\rk_0(\cM):=d$
      for finitely generated $R$-module $\cM$ and then extend it to all $R$-modules via
      $$\rk(\cM):=\sup_{\cN \subseteq \cM} \rk_0(\cN)$$
      for $\cN$ ranging over all finitely generated $R$-submodules of $\cM$. It is a length function since it actually equals $\dim_Q(Q\otimes_R \cM)$ for the  field of fraction $Q$ of $R$. In particular, as $R=\bZ$, the notation $\rk$ recovers as the definition of rank for abelian groups. Note that $\nu(\bQ)=+\infty$ and $\rk(\bQ)=1$ as a $\bZ$-modules.\\
      \item Let $\Gamma$ be any discrete group and $R=\LG$ the left group von Neumann algebra of $\Gamma$. The canonical trace $\tr_\LG \colon R \to \mathbb{C}$ sending $T$ to $<T\delta_e, \delta_e>$ naturally extends to $M_n(R)$. The {\it von Neumann dimension} of any finitely generated projective $R$-module $\mathbb{P}$ is then defined as
      $$\dim_\LG'(\mathbb{P}):=\tr_\LG(P)$$
      where $P \in M_n(R)$ is some idempotent element such that $\mathbb{P} \cong (\LG)^nP$ as $R$-modules. The extended notion {\it von Neumann-L\"{u}ck dimension}  on all $R$-modules \cite[Definition 6.6]{L02} is defined via
      $$\dim_\LG(\mathbb{M}):=\sup_{\mathbb{P}} \dim_{\LG}'(\mathbb{P})$$
      where $\mathbb{P}$ ranges over all finitely generated projective $R$-submodules of $\mathbb{M}$.
       The deep properties of this function, as investigated by L\"{u}ck, imply that $\dim_\LG$ is a length function with $\dim_\LG(R)=1$ \cite[Theorem 6.7]{L02}.
    \end{enumerate}
\end{example}

\subsection{The Sylvester rank function and its extension}

The Sylvester rank functions were introduced by Malcolmson on either finitely presented modules or rectangular matrices \cite{M80}. We now recall its definition in the module-theoretic setting. 
\begin{definition}
 Let $R$ be a unital ring.   A map $\dim$ that assigns to each finitely presented $R$-module a value in $\bR_{\geq 0}$ is called a {\bf Sylvester $R$-module rank function} if the following conditions hold: 
 \begin{enumerate}
     \item $\dim(0)=0, \dim(R)=1$;
     \item $\dim(\cM_1\oplus \cM_2)=\dim(\cM_1)+\dim(\cM_2)$;
     \item $\dim(\cM_3)\leq \dim(\cM_2) \leq \dim(\cM_1)+\dim(\cM_3)$ holds for any short exact sequence 
     $$0\rightarrow \cM_1\rightarrow \cM_2\rightarrow \cM_3\rightarrow 0$$ of 
     finitely presented $R$-modules.
 \end{enumerate}
\end{definition}
Clearly each length function $\rL$ on $R$-modules with $\rL(R)=1$ restricts to a Sylvester module rank function. 
\begin{example} \label{Sylvester eg}
    \begin{enumerate}
        \item Let $R=\ZG$ be the group ring of a discrete group $\Gamma$. The {\it von Neumann-L\"{u}ck rank} on all $R$-modules \cite[Section 2.2]{LL18} is defined via
        $${\rm vrk}(\cM):=\dim_\LG (\LG\otimes_{\ZG}\cM).$$
        This rank ${\rm vrk}$ restricts to a Sylvester $R$-module rank function, but it is not a length function as $\Gamma$ is nonamenable.
        \item Let $R$ be a unital $C^\ast$-algebra with a tracial state $\tau$. The trace $\tau$  extends to $M_n(R)$ for all $n \in \bN$ by $\tau(A):=\sum_{i=1}^n \tau(A_{ii})$. For any finitely presented $R$-module $\cM$, define
        $$\dim(\cM)=n-\lim_{k \to +\infty} \tau(|A|^{1/k}),$$
        where $A \in M_{m, n}(R)$ is some matrix such that  $\cM \cong R^n/R^mA$. Then $\dim$ defines a Sylvester  $R$-module rank function \cite[Proposition 2.4]{JL21}.
    \end{enumerate}
\end{example}
To extend the domain of Sylvester module rank functions on all $R$-modules \cite[Definition 3.1]{L21}, its bivariant version is defined as follows:
\begin{definition}
    A map $\dim(\cdot|\cdot)$ on the collection of all pairs of $R$-modules $\cM_1 \subseteq \cM_2$ with the value in $\bR_{\geq 0}\cup \{+\infty\}$ is called a {\bf bivariant Sylvester $R$-module rank function} if the following hold:
\begin{enumerate}
    \item $\dim(\cdot|\cdot)$ is an isomorphism invariant;
    \item putting $\dim(\cM):=\dim(\cM|\cM)$, one has $\dim(0)=0$ and $\dim(R)=1$;
    \item {\bf direct sum}: for any $R$-modules $\cM_1 \subseteq \cM_2$ and $\cM_1' \subseteq \cM_2'$, one has
    $$\dim(\cM_1 \oplus \cM_1'|\cM_2\oplus \cM_2')=\dim(\cM_1|\cM_2)+\dim(\cM_1'|\cM_2');$$
    \item {\bf upper continuity}: $\dim(\cM_1|\cM_2)=\sup_{\cM_1'} \dim(\cM_1'|\cM_2)$ for $\cM_1'$ ranging over all finitely generated $R$-submodules  of $\cM_1$; 
    \item {\bf lower continuity}: if $\cM_1$ is finitely generated, then
    $$\dim(\cM_1|\cM_2)=\inf_{\cM_2'} \dim(\cM_1|\cM_2')$$
    for $\cM_2'$ ranging over all finitely generated $R$-submodules  of $\cM_2$ containing $\cM_1$;
    \item {\bf alternative additivity}: $\dim(\cM_2)=\dim(\cM_1|\cM_2)+\dim(\cM_2/\cM_1)$. 
\end{enumerate}
\end{definition}

\begin{example}
    \begin{enumerate}
        \item For  the {\it von Neumann-L\"{u}ck rank} ${\rm vrk}$  in Example \ref{Sylvester eg} (1), its bivariant version is defined as
        $${\rm vrk}(\cM_1|\cM_2) :=\dim_{\LG} (\Ima 1\otimes i)$$
        for any $\ZG$-modules $\cM_1 \subseteq \cM_2$ (see \cite[Definition 7.1]{LL19}). Here $i \colon \cM_1 \to \cM_2$ is the inclusion map and $1\otimes i \colon \LG\otimes \cM_1 \to \LG\otimes \cM_2$ is the induced natural map.
        According to \cite[Lemma 7.8]{LL19}, this construction yields a bivariant Sylvester $\ZG$-module rank function.
        \item Let $\Gamma$ be a sofic group with a sofic approximation sequence $\Sigma$. Consider a length function $\rL$ on $R$-modules for some unital ring $R$ such that $\rL(R)=1$.  Using the  sofic approximation sequence $\Sigma$ and the length function $\rL$, a class of bivariant Sylvester $\RG$-module rank functions can be constructed as detailed in \cite[Section 3.1]{LL19}. These functions are shown to have significant connections to dynamical invariants and $L^2$-invariants. 
    \end{enumerate}
\end{example}

The following establishes a close relation between a Sylvester module rank function and its bivariant version \cite[Theorem 3.3]{L21}.
\begin{theorem} \label{Sylvester extension}
    Every Sylvester $R$-module rank function $\dim(\cdot)$  extends uniquely to a bivariant Sylvester $R$-module rank function $\dim(\cdot|\cdot)$ in the sense that 
    $$\dim(\cM)=\dim(\cM|\cM)$$
    for any finitely presented $R$-module $\cM$.    
\end{theorem}

We collect some important properties of bivariant Sylvester module rank functions in the following theorem, as established in \cite[Theorem 2.3]{JL21}.
\begin{theorem} \label{deep S-property}
Let $\dim(\cdot\,|\,\cdot)$ be a bivariant Sylvester $R$-module rank function. The following are true.
\begin{enumerate}
\item For any $R$-modules $\mathcal{M}_1 \subseteq \mathcal{M}_2 \subseteq \mathcal{M}_3$, one has
\[
\dim(\mathcal{M}_2\,|\,\mathcal{M}_3) = \dim(\mathcal{M}_1\,|\,\mathcal{M}_3) + \dim\big(\mathcal{M}_2/\mathcal{M}_1 \,\big|\, \mathcal{M}_3/\mathcal{M}_1\big).
\]

\item For any $R$-modules $\mathcal{M}_1, \mathcal{M}_2 \subseteq \mathcal{M}_3$, one has
\[
\dim(\mathcal{M}_1 + \mathcal{M}_2\,|\,\mathcal{M}_3) + \dim(\mathcal{M}_1 \cap \mathcal{M}_2\,|\,\mathcal{M}_3) \le \dim(\mathcal{M}_1\,|\,\mathcal{M}_3) + \dim(\mathcal{M}_2\,|\,\mathcal{M}_3).
\]

\item For any  $R$-modules $\mathcal{M}_1 \subseteq \mathcal{M}_2$ and any $R$-module homomorphism $\pi : \mathcal{M}_2 \to \mathcal{M}$, one has
\[
\dim\big(\pi(\mathcal{M}_1) \,\big|\, \pi(\mathcal{M}_2)\big) \le \dim(\mathcal{M}_1\,|\,\mathcal{M}_2).
\]

\item For any $R$-modules $\mathcal{M}, \mathcal{M}_1 \subseteq \mathcal{M}_2$ with $\mathcal{M}_1$ finitely generated, denoting by $\pi_{\mathcal{M}'}$ the homomorphism $\mathcal{M}_2 \to \mathcal{M}_2/\mathcal{M}'$ for every $R$-submodule $\mathcal{M}'$ of $\mathcal{M}_2$, one has
\[
\dim\big(\pi_{\mathcal{M}}(\mathcal{M}_1) \,\big|\, \mathcal{M}_2/\mathcal{M}\big)
= \inf_{\mathcal{M}'} \dim\big(\pi_{\mathcal{M}'}(\mathcal{M}_1) \,\big|\, \mathcal{M}_2/\mathcal{M}'\big).
\]
 for $\cM'$ ranging over all finitely generated R-submodules of $\cM$.
\end{enumerate}
\end{theorem}

\subsection{Amenable  groups}
A finite set $F \in \cF(\Gamma)$ is called {\bf $(K, \delta)$-invariant}  for some $K\in \cF(\Gamma)$ and $\delta > 0$ if the following holds
$$|\{s \in F: Ks \subseteq F\}|\geq (1-\delta)|F|.$$
 The group $\Gamma$ is called {\bf amenable} if it is $(K, \delta)$-invariant  for any $K\in \cF(\Gamma)$ and  any $\delta > 0$. 
For a function $\varphi \colon \cF(\Gamma) \to \bR$, we say that   {\bf $\varphi(F)$ converges to some $c\in \bR$ as $F\in \cF(\Gamma)$ gets more and more left invariant}, written as
$$\lim_F\varphi(F)=c,$$
if for any $\varepsilon>0$ there are some $K \in \cF(\Gamma)$ and $\delta > 0$ such that  
$$|\varphi(F)-c|<\varepsilon$$
 for all $(K, \delta)$-invariant  $F \in \cF(\Gamma)$. 

 \begin{example}
    The integral group $\bZ$ is clearly amenable, as the sequence $\{\{0, 1, \cdots, n\}\}_{n=1}^\infty$ is eventually $(K, \delta)$-invariant for any $K\in \cF(\bZ)$ and $\delta > 0$. Furthermore, any group of subexponential growth is amenable.  The class of amenable groups is closed under  taking subgroups, quotients, extensions, and directed unions.  The Grigorchuk group is a notable group of subexponential growth that sits  outside of the class generated by finite groups and abelian groups, i.e.  the class of elementary amenable groups.  The Basilica group is a notable amenable group that sits outside the class generated by groups of subexponential growth \cite{GZ02, BV05}.  
 \end{example}

The following strong version of Ornstein-Weiss lemma is mainly due to Gromov \cite[1.3.1]{G99m}\cite[Theorem 4.38, Theorem 4.48]{KLb} \cite[Theorem 9.4.1]{C15}.
\begin{lemma} \label{OW}
Let  $\varphi \colon \cF(\Gamma)  \to \bR$ be a function satisfying\\
\begin{enumerate}
    \item $\varphi(Fs)=\varphi(F)$ for all $F \in \cF(\Gamma)$ and $s \in \Gamma$; \\
    \item $\varphi(F_1\cup F_2) \leq \varphi(F_1) + \varphi(F_2)$ for all $F_1, F_2 \in \cF(\Gamma)$.\\
\end{enumerate}
Then the limit $\lim_F\varphi(F)/|F|$ exists. Moreover if $\varphi$ is \textbf{strongly subadditive} in the sense that
$$\varphi(F_1\cup F_2) \leq \varphi(F_1) + \varphi(F_2)-\varphi(F_1\cap F_2)$$
for all $F_1, F_2 \in \cF(\Gamma)$, we have
$$\lim_F\frac{\varphi(F)}{|F|}=\inf_{F \in \cF(\Gamma)} \frac{\varphi(F)}{|F|}.$$
\end{lemma}

Throughout the remainder of this paper, the symbol $\Gamma$ will denote a discrete amenable group, unless stated otherwise.

\section{Weak length function}

Let $R$ be a unital ring.
 Denote by $\cP_R$ the collection of pairs $(A, \cM)$ for any $R$-module $\cM$ and any nonempty subset $A$ of $\cM$. Similarly denote by $\cF_R$ the pairs $(A, \cM)$ such that $A$ is a finite set.
 For any subset $A$ of an $R$-module $\cM$ let $<A>_R$ denote the $R$-submodule generated by $A$.  For another subset $B \subseteq \cM$ and $C \subseteq R$, we write  
 $$A+B:= \{a+b: a \in A, b \in B\}, \mbox{\ and \ } CA:=\{ra:  r \in C \ {\rm and \ } a \in A\}.$$
 A simple observation is that if $0 \in A \cap B$, then $A+B$ contains $A\cup B$. For any abelian group $G$ write $\cF(G)$ for the nonempty finite subsets of $G$ and $\cF^0(G)$ for the finite subsets of $G$ containing zero element respectively.

 \begin{definition} \label{weak length def}
 An {\bf $R$-weak length}  is a map $\ell \colon \cP_R \to \bR_{\geq0} \cup \{+\infty\} $  sending $(A, \cM)$ to $\ell(A|\cM)$ such that the following properties hold:
\begin{enumerate}
    \item {\bf regularity}: $\ell(\{0\}|\cM)=0$ for any $R$-module $\cM$;
    \item {\bf  product property}:  $\ell(A\times B|\cM\times \cN)=\ell(A|\cM)+\ell(B|\cN)$,  for  any  $(A,\cM),  (B, \cN) \in \cP_R$;
    \item {\bf quotient property}: for any $R$-module homomorphism $\varphi  \colon \cM \to \cN$, one has $\ell(A|\cM      
    ) \geq \ell(\varphi(A)|\cN)$ for any $(A, \cM) \in \cP_R$;
    \item {\bf upper continuity:} if $A_n$ increases to $A$ for some $R$-module $\cM$ containing all $A_n$ and $A$, then $\ell(A_n|\cM)$ increases to $\ell(A|\cM)$. In particular, $\ell$ is monotone for increasing subsets in the sense that whenever $A_1 \subseteq A_2\subseteq \cM$ one has $\ell(A_1|\cM) \leq \ell(A_2|\cM)$.
\end{enumerate}
We say $\ell$ is {\bf locally finite} if $\ell(A|\cM) < +\infty$ for any $(A, \cM)  \in \cF_R$; $\ell$ has the {\bf strong quotient property} if for  any $R$-module homomorphism $\varphi \colon \cM \to \cN$, one has
$$\ell(A+ B|\cM) \geq \ell(\varphi(A)|\cN)+\ell(B|\cM)$$
for any $ A \in  \cF^0(\cM)$ and $B \in \cF^0(\ker \varphi)$. We say $\ell$ is {\bf independent of ambient modules} if
$$\ell(A|\cM_1)=\ell(A|\cM_2)$$
for any $A \subseteq \cM_1 \subseteq \cM_2$ of $R$-modules $\cM_1$ and $\cM_2$. 
When $\ell$ is partially defined on a subcollection $\cC_R \subseteq \cP_R$ that is closed under taking product, quotient and direct union operations,  we would say $\ell$ is a weak length function on $\cC_R$.
 \end{definition} 

 \begin{remark}
In \cite[Section 1]{SZ09}, Scale-Zanardo used some weaker axioms to introduce the notion of  subadditive invariants on modules. In contrast our approach to weak length functions is to relax the domain in order to facilitate the definition of algebraic entropy and validity of addition formula.  
 \end{remark}

\begin{example} \label{wlexample}
   \begin{enumerate}
       \item A typical example of locally finite $\bZ$-weak length functions we bear in mind is 
   $$\ell(A|\cM):=\log|A|.$$
    Note that  $\log |<A>_\bZ|=+\infty$ whenever $A$ contains an element of infinite order. Such an $\ell$ satisfies the strong quotient property. In contrast, in the setting of strong quotient property, the inequality
    $$\ell(A\cup B|\cM) \geq \ell(\varphi(A)|\cN) +\ell(B|\cM)$$
    does not hold. For example, let $\cM=\cN=\bZ^2$, and let  $\varphi \colon \cM \to \cN$ be the projection onto the first coordinate. Consider the set $A=\{(0,0), (0,1), (1,0)\}$ and $ B=\{(0,0), (0,1)\}$. We observe that $B \in \cF^0(\ker \varphi)$ and 
    $$\ell(A\cup B|\cM)=\log 3 < \log 2 +\log 2=\log|\varphi(A)|+\log|B|.$$
   \item  Let  $\rL$ be a length function on $R$-modules. For any subset $A \subseteq \cM$ of an  $R$-module $\cM$, define
   $$\ell(A|\cM):=\rL(<A>_R),$$
   Then it is direct to verify that such an  $\ell$ is a weak length function on $\cP_R$. Moreover, $\ell$ is locally finite on the collection of subsets of locally $\rL$-finite $R$-modules.
   
    \item For every $k \in \bN$  and an abelian group $\cM$ write
    $T_k(\cM):=\{x \in \cM: kx=0\}$
    for $k$-torisons of $\cM$. Consider 
    $$\ell_k(A|\cM):= \log |T_k(\cM)\cap A|$$
    for any $A \subseteq \cM$.
    Then $\ell_k$ is a weak length function on $\cP_\bZ$ but not a length function on $\bZ$-modules!

  \item Consider $R=\bZ F_2$ for the free group $F_2$ with two canonical generators $a$ and $b$. Using the notion of sofic mean rank $\mrk_\Sigma$ in [Section 7, LL19], for any subset 
  $A \subseteq \cM$ of an $R$-module $\cM$,  define
  $$\ell(A|\cM):=\mrk_\Sigma(<A>_R|\cM).$$
In light of [Example 6.3, LL19], for $A =\{a-1, b-1\} \subseteq \cM_1=R(a-1)\oplus R(b-1)$ and $\cM_2=R$,   we can compute that $\ell(A|\cM_1)=\mrk_\Sigma (\cM_1)=2$; in contrast,
$$\ell(A|\cM_2)=\mrk_\Sigma(\cM_1|\cM_2)=\mrk_\Sigma(\cM_2)-\mrk_\Sigma(\cM_2/\cM_1)=1-0=1.$$
Thus, $\ell$ do depend on the ambient modules!
  
    \item For any subset $A$ of an abelian group, consider ${\rm gen}(A)$ as the  minimal number of generators for the abelian group $<A>_\bZ$. By the structure theorem of finitely generated abelian groups, $\gen(A)$ equals the sum of the rank of free part of $<A>_\bZ$ and the number of invariant factors of $<A>_\bZ$.  The function gen is monotone. However, it does not define a weak length 
    $$\ell(A|\cM):={\rm gen}(A)$$
    as it does not satisfy the product property. This is illustrated by the following example: $$\gen(\bZ/2\bZ\oplus \bZ/3\bZ)=\gen(\bZ/6\bZ)=1 <2 =\gen(\bZ/2\bZ)+ \gen(\bZ/3\bZ).$$

    \item Let $R$ be a valuation domain.  For any subset $A$ of an $R$-module $\cM$, motivated by Malcev rank,  we can consider
     $${\rm Mr}(A|\cM):=\sup_{B \in \cF(<A>_R)} \gen(B).$$
 In \cite[Example 1.7]{SZ09} it is shown that ${\rm Mr}$ defines a weak length on $\cP_R$.
    
   \end{enumerate}
\end{example}
There are more interesting examples in \cite{SZ09}, which can be alternatively referred to as weak length functions. The following proposition collects some basic properties of weak length functions.

\begin{proposition} \label{basic prop}
An  $R$-weak length function $\ell$ satisfies the following:
\begin{enumerate}
    \item $\ell(A+B|\cM) \leq \ell(A|\cM)+\ell(B|\cM)$, for any $A, B\subseteq \cM$ of an $R$-module $\cM$; 
    \item if $A \subseteq \cM_1 \subseteq \cM_2$, then $\ell(A|\cM_1) \geq \ell(A|\cM_2)$;
    \item if $0 \in A\cap B$, then $\ell((A\cup B|\cM) \leq \ell(A+B|\cM)$;
    \item if $\varphi \colon \cM \to \cN$ is an isomorphism of $R$-modules, then $$\ell(A|\cM)=\ell(\varphi(A)|\cN)$$ for any nonempty $A \subseteq \cM$; This means that $\ell$ has the invariance property.
\end{enumerate}   
\end{proposition}

\begin{proof}
    By the product and quotient property of the weak length function, (1) follows from the homomorphism $ \cM \times \cM \to \cM$ sending $(a, b)$ to $a+b$. 
    
     (2) follows from the quotient property of the weak length function. (3) follows from the monotonicity of the weak length function. 
    (4) follows from
    $$\ell(A|\cM) \geq \ell(\varphi(A)|\cN) \geq \ell(\varphi^{-1}(\varphi(A)|\cM)=\ell(A|\cM).$$
\end{proof}

\begin{proposition} \label{length induces sq}
    Let $R$ be a unital ring and $\rL$ a locally-finite length function on $R$-modules.  Consider the induced  weak length defined by
    $$\ell(A|\cM):=\rL(<A>_R).$$ 
Then $\ell$ has the strong quotient property.  
\end{proposition}

\begin{proof}
   Let $\varphi \colon \cM \to \cN$ be an $R$-module homomorphism. Since $\rL$ is additive, for any $A \in \cF^0(\cM), B \in \cF^0(\ker \varphi)$ we have
    \begin{align*}
        \ell(A+B|\cM)& \geq \ell(A\cup B|\cM)\\
        & =\rL(<\varphi(A)>_R) +\rL(\ker \varphi \cap <A, B>_R) \\
        &\geq \rL(<\varphi(A)>_R) +\rL(<B>_R) \\
        &=\ell(\varphi(A)|\cN) +\ell(B|\cM). 
    \end{align*}
$$$$
\end{proof}

\begin{question}
    Under what conditions is a weak length $\ell$ independent of ambient modules?
\end{question}

Now we introduce a bivariant version of weak length functions motivated from \cite[Section 2.2]{DFGS23}.

\begin{definition}
A  {\bf bivariant $R$-weak length function} $\ell(\cdot, \cdot|\cdot)$ is an $\bR_{\geq 0}\cup \{\infty\}$-valued map sending $(A, B, \cM)$ for all nonempty subsets $A, B \subseteq \cM$ of any $R$-module $\cM$ to $\ell(A,B|\cM)$, satisfying  the following properties:
\begin{enumerate}
    \item {\bf regularity:} $\ell(0, 0|\cM)=0$ for any $R$-module $\cM$;
    \item {\bf direct product: } $\ell(A\times A', B\times B'|\cM\times \cM')=\ell(A, B|\cM)+\ell(A', B'|\cM')$  for  any  $A, B \subseteq \cM$ and $A', B' \subseteq \cM'$ for some $R$-modules $\cM $ and $\cM'$;
    \item {\bf quotient property}: for any $R$-module homomorphism  $\varphi  \colon \cM \to \cN$, and nonempty $A, B \subseteq \cM$, one has 
    $$\ell(\varphi(A), \varphi(B)|\cN) \leq \ell(A, B|\cM);$$
    \item {\bf continuity property}: if $A_n \subseteq \cM$ increases to $A$, then $\ell(A_n,B|\cM)$ increases to $\ell(A,B|\cM)$ for any nonempty $B\subseteq \cM$; if $A \subseteq \cM$ is  finite, and $B_n \subseteq \cM$ increases to $B$, then $\ell(A,B_n|\cM)$ decreases to $\ell(A,B|\cM)$.
\end{enumerate}
\end{definition}
In particular, $\ell(\cdot, 0|\cdot)$ is an $R$-weak length.

Some easy consequences of bivariant $R$-weak lengths are as follows:
\begin{proposition} \label{basic bivariant}
Any bivariant $R$-weak length function $\ell(\cdot, \cdot|\cdot)$ satisfies the following properties:
\begin{enumerate}
    \item for any nonempty $A,A',  B, B' \subseteq \cM$ of an $R$-module $\cM$, one has
    $$\ell(A+A', B+B'|\cM) \leq \ell(A,B|\cM)+\ell(A', B'|\cM).$$
    \item if $0 \in A\cap A'$, then $\ell(A\cup A', B|\cM) \leq \ell(A+A', B|\cM)$;
    \item if $\varphi \colon \cM \to \cN$ is an isomorphism of $R$-modules, then $$\ell(A, B|\cM)=\ell(\varphi(A), \varphi(B)|\cN)$$ 
for any nonempty $A, B \subseteq \cM$. This means that $\ell(\cdot, \cdot|\cdot)$ has the invariance property;
\end{enumerate}   
\end{proposition}

\begin{proof}
    By the product and quotient property, (1) follows from the homomorphism 
    $$\varphi \colon (\cM \times \cM)\times (\cM \times \cM) \to \cM\times \cM$$ sending $(\xi,\xi')$ to $\xi+\xi'$.
    (2) follows from the  continuity property. 
    (3) follows from
    $$\ell(A, B|\cM) \geq \ell(\varphi(A), \varphi(B)|\cN) \geq \ell(\varphi^{-1}(\varphi(A)), \varphi^{-1}(\varphi(B))|\cM)=\ell(A, B|\cM).$$
\end{proof}

Now we introduce a key upgrading condition connecting weak lengths and their bivariant versions.
\begin{definition} \label{upgrading def}
    A bivariant $R$-weak length function $\ell(\cdot, \cdot|\cdot)$ is called an  {\bf upgrading of an $R$-weak length function $\ell$}, if it satisfies  
    $$\ell(A, 0|\cM)=\ell(A|\cM)$$
    for any $(A, \cM) \in \cP_R$. We say $\ell(\cdot, \cdot|\cdot)$ is a {\bf proper upgrading}
    if it satisfies the following:
    \begin{enumerate}
        \item[i).] $\ell(A, C|\cM) \leq \ell(A, B|\cM) +\ell(B, C|\cM)$ for any $A, B, C \in \cF(\cM)$; in particular, one has
        $$\ell(A|\cM) \leq \ell(A, B|\cM) +\ell(B|\cM).$$
       \item[ii).]  for any $R$-module homomorphism $\varphi \colon \cM \to \cN$  and $A \in \cF(\cM)$, then for any $\varepsilon > 0$, there exists $B \in \cF(\ker \varphi)$ such that
        $$\ell(A, B|\cM) \leq \ell(\varphi(A)|\varphi(\cM)) +\varepsilon;$$ 
    \end{enumerate}
\end{definition}
in particular,  it implies that $\ell(\varphi(A)|\cN)\leq \ell(A, B|\cM) \leq \ell(\varphi(A)|\varphi(\cM))+\varepsilon$.

An important enquiry is to ask:
\begin{question}
    When can an $R$-weak length function be upgraded to a proper bivariant $R$-weak length function?
\end{question}

\begin{example} \label{grading eg}
    \begin{enumerate}
        \item Consider the weak length $\ell$ induced from a length function $\rL$ as in Example \ref{wlexample}(2). Clearly such a weak length $\ell$ is independent of ambient modules. For any $A, B \subseteq \cM$ of an $R$-module $\cM$, write $\pi_B$  for the quotient map $\cM \to \cM/<B>_R$. Define $\ell(\cdot, \cdot|\cdot)$ via
$$\ell(A, B|\cM):=\ell(\pi_B(A)|\pi_B(\cM))=\rL(<\pi_B(A)>_R).$$
By the additivity of $\rL$, we have
$$\ell(A\cup B|\cM)=\ell(A, B|\cM)+\ell(B|\cM).$$
\item For a Sylvester $R$-module rank function $\dim(\cdot)$, by Theorem \ref{Sylvester extension}, it uniquely extends to a bivariant Sylvester $R$-module rank function $\dim(\cdot|\cdot)$.
For any $A, B \subseteq \cM$ of an $R$-module $\cM$, define 
$$\ell(A, B|\cM):=\dim(\pi_B(<A>_R)|\pi_B(\cM))$$
and
$$\ell(A|\cM):=\ell(A, 0|\cM)=\dim(<A>_R|\cM).$$
In light of Example \ref{wlexample} (4), we see that such a function $\ell$ can be dependent on ambient modules.
By the alternative addition formula of $\dim(\cdot|\cdot)$, we also have
$$\ell(A\cup B|\cM)=\ell(A, B|\cM)+\ell(B|\cM).$$
It follows that such an $\ell$ has the strong quotient property.
    \end{enumerate}
    
\end{example}

\begin{proposition} \label{length upgrading}
    Let $\ell(\cdot, \cdot|\cdot)$ be a function defined above from  a locally-finite length function $
    \rL$ or a bivariant Sylvester module rank function on $R$-modules. Then $\ell(\cdot, \cdot|\cdot)$ is a proper bivariant $R$-weak length function that upgrades $\ell$. Precisely the following hold:
\begin{enumerate}
    \item  for any $A, B \subseteq \cM$ of an $R$-module $\cM$ and $C \in \cF(\cM)$, we have
    $$\ell(A, C|\cM) \leq \ell(A\cup B, C|\cM)\leq \ell(A, B|\cM)+\ell(B, C|\cM).$$
   \item  let $\varphi \colon \cM \to \cN$ be any R-module homomorphism, $A \in \cF(\cM)$, then for any $\varepsilon > 0$, there exists $B \in \cF(\ker \varphi)$ such that
   $$\ell(A, B|\cM) \leq \ell(\varphi(A)|\varphi(\cM)) +\varepsilon.$$
\end{enumerate}
\end{proposition}

\begin{proof}
 We consider the case that $\ell$ is induced from a bivariant Sylvester module rank function $\dim(\cdot|\cdot)$. The case for the locally-finite lengths can be similarly deduced in an easier way. 

 First we verify that $\ell(\cdot, \cdot|\cdot)$ is a bivariant $R$-weak length function. It is clear that $\ell((0, 0)|\cM)=0$ for any $R$-module $\cM$. For any nonempty $A, B \subseteq \cM$ and $A', B' \subseteq \cM'$ of some $R$-modules $\cM$ and $\cM'$, using the direct sum property of $\dim(\cdot|\cdot)$, we have
 \begin{align*}
     &\ell(A\times A', B\times B'|\cM\times \cM')\\
     &=\dim(\pi_{B\times B'}(<A\times A'>_R)|\pi_{B\times B'}(\cM\times \cM')) \\
     &=\dim(\pi_B(<A>_R)\oplus \pi_{B'}(<A'>_R)|\pi_B(\cM)\oplus \pi_{B'}(\cM'))\\
     &=\dim(\pi_B(<A>_R)|\pi_B(\cM))+\dim(\pi_{B'}(<A'>_R)|\pi_{B'}(\cM'))\\
     &=\ell(A,B|\cM) + \ell(A',B'|\cM').
 \end{align*}

To verify the quotient property, for any $R$-module homomorphism $\varphi \colon \cM \to \cN$ and $A, B \subseteq \cM$, it induces an $R$-module homomorphism $\varphi' \colon \cM/<B>_R \to \cN/<\varphi(B)>_R$ sending $x+<B>_R$ to $\varphi(x)+<\varphi(B)>_R$ and $\pi_{\varphi(B)}\circ \varphi =\varphi' 
\circ \pi_B$. Applying Theorem \ref{deep S-property} (3), we have
\begin{align*}
    \ell(\varphi(A), \varphi(B)|\cN)&=\dim(\pi_{\varphi(B)}(\varphi(A))|\pi_{\varphi(B)}(\cN))\\
    &=\dim(\varphi'(\pi_B(A))|\varphi'(\pi_B(\cN)))\\
    &\leq \dim(\pi_B(A)|\pi_B(\cN))=\ell(A, B|\cN).
\end{align*} 
The upper continuity of $\ell$ easily follows from the upper continuity of $\dim(\cdot|\cdot)$. The lower continuity of $\ell$ follows from Theorem \ref{deep S-property} (4). 

To show that $\ell$ has the strong quotient property, we consider an $R$-module homomorphism $\varphi \colon \cM \to \cN$, $A \in \cF^0(\cM)$ and $B \in \cF^0(\ker \varphi)$. Note that $\varphi $ factors through $\pi_B$. By Theorem \ref{deep S-property} (3), we have 
$$\dim(\varphi(A)|\cN) \leq \dim(\pi_B(A)|\pi_B(\cM)).$$
Applying the alternative addition formula of $\dim$ to $R$-modules  $<B>_R \subseteq <A+B>_R \subseteq \cM$, we obtain
\begin{align*}
    \ell(A+B|\cM)&=\dim(<A+B>_R|\cM)\\
    &=\dim(<B>_R|\cM) +\dim(\pi_B(<A>)|\pi_B(\cM))\\
    &\geq \dim(<B>_R|\cM) +\dim(\varphi(A)|\cN)\\
    &=\ell(B|\cM) +\ell(\varphi(A)|\cN).
\end{align*}

Now we verify that $\ell(\cdot, \cdot)$ is a proper upgrading. Using $\dim(R)=1< \infty$, for any $C \in \cF(\cM)$ of an $R$-module $\cM$, we have 
$$\ell(C|\cM)=\dim(<C>_R|\cM) \leq \dim(<C>_R|<C>_R) \leq |C|\dim(R)< \infty.$$
This shows that $\ell$ is locally finite.  Then (1) follows from 
   \begin{align*}
    \ell(A\cup B, C|\cM )&=\ell(A\cup B\cup C|\cM) -\ell(C|\cM)\\
    &=\ell(A\cup B\cup C, B\cup C|\cM) +\ell((B\cup C)|\cM) -\ell(C|\cM)\\
    &= \ell(A, B\cup C|\cM)+\ell(B, C|\cM)\\
    &\leq \ell(A, B|\cM) +\ell(B, C|\cM).
\end{align*} 

For (2), denote by $\overline{\varphi}$ the isomorphism $\cM/\ker \varphi \to \cN$ sending $x+\ker \varphi$ to $\varphi(x)$. Put  $\mathcal{K}=\ker \varphi$. Then $\varphi=\overline{\varphi} \circ \pi_\mathcal{K}$. By the invariance property, we have 
$$\dim(\varphi(<A>_R)|\varphi(\cM))=\dim(\pi_\mathcal{K}(<A>_R)|\pi_\mathcal{K}(\cM)).$$
From Theorem \ref{deep S-property} (4),  $\ell(\cdot, \cdot|\cdot)$ is lower continuous. Thus there exists $B \in \cF(\mathcal{K})$ such that
$$\dim(\pi_B(<A>_R)|\pi_B(\cM)) < \dim(\pi_\mathcal{K}(<A>_R)|\pi_\mathcal{K}(\cM)) +\varepsilon. $$
It follows that
\begin{align*}
    \ell(A, B|\cM)&=\dim(\pi_B(<A>_R)|\pi_B(\cM))\\
    &< \dim(\varphi(<A>_R)|\varphi(\cM))+\varepsilon =\ell(\varphi(A)|\varphi(\cM))+\varepsilon.
\end{align*}

\end{proof}

Another important upgrading situation for the $\bZ$-weak length function
$\ell(A|\cM):=\log|A|$
(see \cite[Section 2.2]{DFGS23}) is the following:  
$$\ell(A, B|\cM):=\min \{ \ell(C): A \subseteq C+B \mbox{ \ for \ some \ } C \subseteq \cM \}.$$

The following properties are discussed in \cite[Section 2.2]{DFGS23}. For completeness, we provide a slightly short proof.

\begin{proposition} \label{log upgrading}
   The $\bZ$-weak length function $\ell(\cdot|\cM):=\log|\cdot|$ has the strong quotient property and has a proper upgrading as above.  Precisely the following hold:
\begin{enumerate}
    \item  for any $A, B \subseteq \cM$ of a $\bZ$-module $\cM$ and $C \in \cF(\cM)$, we have
    $$\ell(A, C|\cM) \leq  \ell(A, B|\cM)+\ell(B, C|\cM).$$
In particular, $\ell(A|\cM)\leq \ell(A,B|\cM)+\ell(B|\cM)$. 
\item for any $A, B \subseteq \cM_1 \subseteq \cM_2$ of $\bZ$-modules $\cM_1$ and $\cM_2$, one has
$\ell(A,B|\cM_1) \geq \ell(A, B|\cM_2)$.
   \item  for any $\bZ$-module homomorphism $\varphi \colon \cM \to \cN$ and $A \in \cF(\cM)$, there exists $B \in \cF^0(\ker \varphi)$ such that
   $$\ell(A, B|\cM) = \ell(\varphi(A)|\cN).$$
\end{enumerate}
    
\end{proposition}

\begin{proof}
(1) and (2) are clear by definition.  For (3),  put $B=(A-A)\cap \ker \varphi$. We shall prove $B$ is the desired. Choose $A_0\subseteq A$ such that $|A_0|=|\varphi(A_0)|=|\varphi(A)|$. In particular, $\varphi$ is injective when restricted on $A_0$. By definition of $B$ and choice of $A_0$, we have 
 $$A \subseteq \bigsqcup_{a \in A_0} (a+B)$$
 and $A\cap (a+B)\neq \emptyset$ for every $a \in A_0$. It follows that 
 $$\ell(A, B|\cM) \leq \ell(A_0|\cM)=\ell(\varphi(A)|\cN).$$               

 Next we show $\ell(\varphi(A)|\cN) \leq \ell(A,B|\cM)$. Assume $\ell(A,B|\cM) < \ell(\varphi(A)|\cN)$. By definition, there exists $C \in \cF(\cM)$ such that 
 $$A \subseteq C+B {\rm  \ and \ }\log |C|=\ell(A, B|\cM)< \ell(\varphi(A)|\cN)=\log|A_0|.$$  For every \( a \in A_0 \), since $A\cap (a+B)\neq \emptyset$ and $A \subseteq C+B$,  there exists some \( c_a \in C \) such that \( (c_a + B) \cap (a + B) \neq \emptyset \). This provides us a map $f \colon A_0 \to C$ sending $a$ to $c_a$. Since \( |A_0| > |C| \), it forces that $f$ is not injective. Thus there exists $a\neq a' \in A_0$ such that 
$$(a+B)\cap (c+B)\neq \emptyset {\rm \ and  \ } (a'+B)\cap (c+B)\neq \emptyset.$$
It concludes that $0\neq a-a' \in (B-B)+(B-B) \subseteq \ker \varphi$, which contradicts the choice of $A_0$. 

To see the strong quotient property, consider an injective map $\varphi(A)\times B \to A_0+B$ sending $(y, b)$  to $a_y +b$ where $a_y \in A_0$ is the unique element such that $\varphi(a_y)=y$. Then
$$\ell(A+B|\cM) \geq \ell(A_0+B|\cM) =\ell(\varphi(A)\times B|\cN\times \cM)=\ell(\varphi(A)|\cN)+\ell(B|\cM).$$
\end{proof}

\begin{remark}
For $\ell(A|\cM)=\log|A|$ for any  subset $A \subseteq \cM$ of an abelian group $\cM$, one may consider
$$\ell'(A, B|\cM):=\ell(\varphi_B(A)|\varphi_B(\cM))=\log|\varphi_B(A)|$$
for the quotient map $\varphi_B \colon \cM \to \cM/<B>$ and $A, B \subseteq \cM$.
By definition one can easily infer that 
$$\ell'(A, B|\cM) \leq \ell(A, B|\cM)$$
and the inequality can be strict! Though $\ell'(\cdot, \cdot)$ is an upgrading of $\ell$, it is not proper since the inequality
$$\ell'(A|\cM) \leq \ell'(A, B|\cM)+\ell'(B|\cM)$$
does not hold.
For example, consider $\cM=\bZ^2, B=\{(1, 0)\}, A=\{(1,0), (1,1)\}$. One easily computes that $\ell'(A|\cM)=\log 2$ and $\ell'((A, B)|\cM)=\log 1 =0=\ell'(B|\cM)$. This illustrates that the proper upgrading paths of the weak length functions induced from length functions and $\log|\cdot|$ are different.

On the other hand, when $B$ is a subgroup, one easily concludes that $\ell(A, B|\cM)=\ell(A,<B>|\cM)=\ell'(A,B|\cM)$ (see also \cite[Proposition 2.6 (a)]{DFGS23}).
\end{remark}

\section{Mean weak length function}

In this section we fix a locally finite $R$-weak length function $\ell$ and an amenable group $\Gamma$. We shall introduce the notion of mean weak length based on a given weak length and discuss its basic properties. 

Let $\cM$ be any $\RG$-module. For each nonempty subset $A \subseteq \cM$ and $F \in \cF(\Gamma)$ put
$$A^{[F]}:=\sum_{s \in F} s^{-1}A.$$
\begin{lemma} \label{subadd}
    Let $\cM$ be any $\RG$-module and  $A \in \cF^0(\cM)$.  Consider  a function $\varphi \colon \cF(\Gamma) \cup \{\emptyset\} \to \bR$ sending $F$ to $\ell(A^{[F]}|\cM)$ and $\varphi(\emptyset)=0$. Then $\varphi$ is $\Gamma$-invariant, monotone, and  subadditive.  Moreover, if $\ell$ is induced from a locally finite length function or a Sylvester module rank function,  and $A=-A$, then $\varphi$ is strongly subadditive.
\end{lemma}

\begin{proof}
    By Proposition \ref{basic prop} (4), $\varphi$ is  $\Gamma$-invariant. Since $0 \in A$, we have  $A^{[F_1]} \subseteq A^{[F_2]}$ provided that $F_1 \subseteq F_2$ are nonempty finite subsets of $\Gamma$. By the monotonicity of $\ell$, it follows that $\varphi$ is monotone as well. Consider the homomorphism $\psi \colon <A^{[F_1]}>_R\times <A^{[F_2]}>_R \to \cM$ sending $(x,y)$ to $x+y$. Since $0 \in A$, we have  $A^{[F_1\cup F_2]} \subseteq \Ima \psi$. By the direct product property and quotient property of $\ell$, we get
    $$\ell(A^{[F_1]}|\cM)+\ell(A^{[F_2]}|\cM) =\ell(A^{[F_1]}\times A^{[F_2]}|\cM\times \cM) \geq \ell(\Ima \psi|\cM) \geq \ell(A^{[F_1\cup F_2]}|\cM).$$

  Now we assume $\ell$ is induced from a locally finite length function or a Sylvester module rank function, $A=-A$, and $F_1\cap F_2\neq \emptyset$. Consider the injective homomorphism $i \colon <A^{[F_1\cap F_2]}>_R \to <A^{[F_1]}>_R\times <A^{[F_2]}>_R$ sending $u$ to $(u, -u)$. Since $A=-A$, we have $i(A^{[F_1\cap F_2]}) \subseteq \ker \psi$. From Proposition \ref{length induces sq}  and Example \ref{grading eg}, $\ell$ has the strong quotient property and $\ell(A+B|\cM)=\ell(A\cup B|\cM)$. Thus
  \begin{align*}
      \ell(A^{[F_1]}|\cM)+\ell(A^{[F_2]}|\cM)&=\ell((A^{[F_1]}\times A^{[F_2]})\cup i(A^{[F_1\cap F_2]})|\cM\times \cM)\\
      &=\ell((A^{[F_1]}\times A^{[F_2]})+ i(A^{[F_1\cap F_2]})|\cM\times \cM)\\
      &\geq \ell(\Ima \psi|\cM)+\ell(i(A^{[F_1\cap F_2]})|\cM\times \cM)\\
      &\geq \ell(A^{[F_1\cup F_2]}|\cM) +\ell(A^{[F_1\cap F_2]}|\cM)
  \end{align*}
as desired.  
    
\end{proof}

By Lemmas \ref{OW} and \ref{subadd},  we conclude that for any $A \in \cF^0(\cM)$ the limit
$$\ml(A|\cM):=\lim_F \frac{\ell(A^{[F]}|\cM)}{|F|}$$
exists. Moreover, if $A=-A$ and $\ell$ is induced from a locally finite length function or a Sylvester module rank function, then   $\ml(A|\cM)= \inf_{F \in \cF(\Gamma)}\frac{\ell(A^{[F]}|\cM)}{|F|}$. 

Adapting the definition of mean length \cite[Section 5]{LL19}, we introduce the notion of mean weak length as follows:
\begin{definition} 
    Let $\cM_1 \subseteq \cM_2$ be two $\RG$-modules. The {\bf mean weak length of $\cM_1$ relative to $\cM_2$} is defined by
    $$\ml(\cM_1|\cM_2):=\sup_{A \in \cF^0(\cM_1)} \ml(A|\cM_2).$$
The {\bf mean weak length of $\cM_1$} is defined as
$$\ml(\cM_1):=\ml(\cM_1|\cM_1)$$
\end{definition}
For the $\bZ$-weak length function $\ell(\cdot|\cM):=\log|\cdot|$, the induced mean weak length corresponds to the {\bf algebraic entropy} on 
$\ZG$-modules, which will be conventionally denoted by $h$.
 As the mean weak length is defined based on a locally finite length function, it recovers as the mean length in \cite[Remark 3.16]{LL18}\cite[Section 5]{LL19}. Since both of these two types of weak lengths are independent of ambient modules, the induced mean weak lengths are independent of the ambient modules.

The following is clear from the definition.
\begin{proposition} 
    For any $\RG$-module $\cM$ and $A, B \in \cF^0(\cM)$, we have
    $$\ml(A\cup B|\cM) \leq \ml(A+B|\cM) \leq \ml(A|\cM)+\ml(B|\cM).$$
\end{proposition}

From the monotonicity and product properties of weak length, we can easily conclude that mean weak length is additive for direct sums.
\begin{proposition} \label{sum is additive}
    For any $\RG$-module $\cM_1\subseteq \cM_2$ and $\cM_1' \subseteq \cM_2'$, we have
    $$\ml(\cM_1\times\cM_1'|\cM_2\times\cM_2')=\ml(\cM_1|\cM_2)+\ml(\cM_1'|\cM_2').$$
\end{proposition}

\begin{proof}
    For any $A_1 \in \cF^0(\cM_1)$ and $A_1' \in \cF^0(\cM_1')$, since $\ell$ has the product property, we have
    $$\ml(A_1\times A_1'|\cM_2\times \cM_2')=\ml(A_1|\cM_2)+\ml(A_1'|\cM_2').$$
    In particular, we obtain that $\ml(\cM_1\times \cM_1'|\cM_2\times \cM_2') \geq \ml(\cM_1|\cM_2)+\ml(\cM_1'|\cM_2')$.

    For any $A \in \cF^0(\cM_1\times \cM_1')$, one has $A \subseteq A_1\times A_1'$ where $A_1$ and $A_1'$ are the projections of $A$ into $\cM_1$ and $\cM_1'$ respectively. Since $\ell$ is monotone for increasing subsets, we obtain 
    $$\ml(A|\cM_2\times \cM_2') \leq \ml(A_1\times A_1'|\cM_2\times \cM_2')=\ml(A_1|\cM_2)+\ml(A_1'|\cM_2')$$
    and hence $\ml(\cM_1\times \cM_1'|\cM_2\times \cM_2') \leq \ml(\cM_1|\cM_2)+\ml(\cM_1'|\cM_2')$.
\end{proof}

    \begin{remark}
 \begin{enumerate}
     \item  Distinct from the mean length, the mean weak length does not possess the generating property as in \cite[Lemma 3.10]{LL18}. To see this simply consider $\Gamma=\bZ$ and $\cM=\ZG$. Then $\{\delta_e\}\subseteq \cM$ is a generating finite subset. However, one easily sees that $$h(\{\delta_e\}|\cM)=\lim_{n \to \infty} \frac{\log 1}{n}=0$$
    and $h(\cM)\geq h(<\delta_e>_\bZ|\cM) \geq h(\{[0, k)\delta_e\}|\cM)=\log k$ for any $k \in \bN$.
    \item    Let $\Gamma$ be an amenable containing a subgroup $H$ isomorphic to $\bZ$ such that $[\Gamma: H]=+\infty$, it is direct to compute that 
    $$h(\bZ(\Gamma/H))=0$$
    though $\bZ(\Gamma/H)$ is an abelian group of infinite rank. This stands in sharp contrast to the  $\bZ$-action case.  In  \cite[Lemma 3.3]{DG16}, it is shown that for $\Gamma=\bZ=<s>$, one has $h(\cM)=\infty$ if $\{s^nx: n \in \bN\}$ generates an abelian group of infinite rank for some $x \in \cM$. This fact plays a key role in the reduction step of the proof of addition formula for the case $\Gamma=\bZ$. 
 \end{enumerate}  
\end{remark}

An alternative generating property of mean weak length is as follows:
\begin{proposition} \label{generating prop}
    Let $\cM_1\subseteq \cM_2$ be two $\RG$-modules. Suppose that $\cM_1$ has a generating subset $A$ in the sense that $<A>_{R\Gamma}=\cM_1$. Then 
    $$\ml(\cM_1|\cM_2)=\ml(<A>_R|\cM_2):=\sup_{A' \in \cF^0(<A>_R)} \ml(A'|\cM_2).$$
\end{proposition}

\begin{proof}
    Let $B=\{b_1, \cdots, b_n\} \in  \cF^0(\cM_1)$ be any finite subset. Then $B \subseteq <KA>_R$ for some $K \in \cF(\Gamma)$. Write each $b_j$ as 
    $$b_j=\sum_{t \in K} r_{j, t} ta_t$$
    for some $r_{j, t} \in R$ and $a_t \in A$. Put $C=\{r_{j,t}: j=1,\cdots, n, t \in K\}$. 
    Then it is easy to verify that $B^{[F]} \subseteq (CA)^{[K^{-1}F]}$. For any $F \in \cF(\Gamma)$, it follows that
$$\frac{\ell(B^{[F]}|\cM_2)}{|F|} \leq \frac{\ell((CA)^{[K^{-1}F]}|\cM_2)}{|F|}.$$
    
 By Lemma \ref{subadd} and amenability of $\Gamma$, we conclude that   
 $$\ml(B|\cM_2) \leq \ml(CA|\cM_2) \leq \ml(<A>_R|\cM_2).$$
Thus $\ml(\cM_1|\cM_2) \leq \ml(<A>_R|\cM_2)$ and  the desired equality follows.
\end{proof}

\begin{proposition} \label{tensor case}
    Let $\cN_1\subseteq \cN_2$ be two $R$-modules. Then for any finite set $A \subseteq \cN_1$, one has
    $$\ml(\RG\otimes_R<A>_R|\RG\otimes_R \cN_2)=\ell(<A>_R|\cN_2).$$
    In particular, $\ml(\RG\otimes_R\cN_1|\RG\otimes_R \cN_2)=\ell(\cN_1|\cN_2)$.
\end{proposition}

\begin{proof}
    Since $\delta_e\otimes A$ generates $\RG\otimes_R<A>_R$ as $\RG$-modules, by Proposition \ref{generating prop}, we have
    $$\ml(\RG\otimes_R<A>_R|\RG\otimes \cN_2)=\sup_{B \in \cF^0(<A>_R)} \ml(\delta_e\otimes B|\RG\otimes_R\cN_2).$$

  Note that $\RG\otimes_R\cN_2$ is isomorphic to $\oplus_{s \in \Gamma}(R\delta_{s^{-1}}\otimes \cN_2)$ as $R$-modules. For each $B \in \cF^0(<A>_R)$ and $F \in \cF(\Gamma)$, we have $\sum_{s \in F} s^{-1}(\delta_e\otimes B)=\oplus_{s \in F} (\delta_{s^{-1}}\otimes B)$ and 
  $$\oplus_{s \in \Gamma}(R\delta_{s^{-1}}\otimes \cN_2)=\oplus_{s \in F}(R\delta_{s^{-1}}\otimes \cN_2)\times\oplus_{s \in \Gamma \setminus F}(R\delta_{s^{-1}}\otimes \cN_2).$$
 Using $\ell(0|\cM)=0$  for any $R$-module $\cM$ and the direct product property of $\ell$, we have
 \begin{align*}
    &\ell(\sum_{s \in F} s^{-1}(\delta_e\otimes B)|\RG\otimes_R\cN_2)\\
    &=\ell(\oplus_{s \in F} (\delta_{s^{-1}}\otimes B)|\oplus_{s \in F}(R\delta_{s^{-1}}\otimes \cN_2)\times\oplus_{s \in \Gamma \setminus F}(R\delta_{s^{-1}}\otimes \cN_2))\\
    &=\ell(\oplus_{s \in F} (\delta_{s^{-1}}\otimes B)|\oplus_{s \in F}(R\delta_{s^{-1}}\otimes \cN_2))\\
    &=|F|\ell(B|\cN_2).
 \end{align*}
 Dividing both sides by $|F|$ and taking limit along F{\o}lner sets, we obtain
 $$\ml(\delta_e\otimes B|\RG\otimes \cN_2)=\ell(B|\cN_2).$$
 By the continuity of $\ell$, the first equality then follows.

 Note that for any finite set $A_0 \subseteq \RG\otimes \cN_1$, there exists finite $A \in \cF^0(\cN_1)$ such that $A_0 \subseteq \RG\otimes_R<A>_R$. Thus
 \begin{align*}
     &\ml(\RG\otimes_R\cN_1|\RG\otimes_R \cN_2)\\
     &=\sup_{A \in \cF^0(\cN_1)} \ml(\RG\otimes_R<A>_R|\RG\otimes_R \cN_2)\\
     &=\sup_{A \in \cF^0(\cN_1)} \ell(<A>_R|\cN_2)=\ell(\cN_1|\cN_2).
 \end{align*}
\end{proof}

\begin{remark}
    Let $R=\bZ F_2$ and $\ell$ be the weak length in Example \ref{wlexample} (4). Consider $\cN_1:=R(a-1)\oplus R(b-1) \subseteq R:=\cN_2$. Then we have $\RG\otimes_R \cN_1 \subseteq \RG\otimes_R\cN_2$ as $\RG$-modules. Note that $\cN_1\cong R^2$ as $R$-modules. By Proposition \ref{tensor case}, we have
    $$\ml(\RG\otimes \cN_1)=\ell(\cN_1|\cN_1)=2\ell(R|R)=2>1=\ell(R|R)=\ml(\RG\otimes \cN_2).$$
    This illustrates that mean weak length can also fail the monotonicity!
\end{remark}

 For  any group $\Gamma$, we can treat the product  $R^\Gamma$ as an $\RG$-module via the  left shift action of $\Gamma$  on $R^\Gamma$, i.e. $(sx)(t)=x(s^{-1}t)$ for any  $x \in R^\Gamma$ and $s, t \in \Gamma$. Endow ${\rm Hom}_R(\RG, R)$ with a $\RG$-module structure via
 $$(a\varphi)(x):=\varphi(xa)$$
 for any $\varphi \in {\rm Hom}_R(\RG, R)$ and $a, x \in \RG$. Then we have an isomorphism between $R^\Gamma$ and ${\rm Hom}_R(\RG, R)$ as $\RG$-modules.
\begin{corollary} \label{basic eg}
      For any $R$-weak length $\ell$, one has $\ml(R\Gamma|R^\Gamma)=\ell(R|R)$. 
\end{corollary}

\begin{proof}
By Proposition \ref{generating prop}, we have $$\ml(R\Gamma|R^\Gamma)=\ml(<\delta_e>_R|R^\Gamma)=\ml(R\delta_e|R^\Gamma).$$
For any $A \in \cF^0(R)$ and any $F \in \cF(\Gamma)$, since $\sum_{s \in F}s^{-1}A=\oplus_{s \in F} s^{-1}A$, using $\ell(0|\cM)=0$ and the direct product property of $\ell$, we have
$$\ell(\sum_{s \in F} s^{-1}A|R^\Gamma)=\ell(\sum_{s \in F} s^{-1}A|R^{F^{-1}}\times R^{\Gamma\setminus F^{-1}})=\sum_{s \in F} \ell(s^{-1}A|R^{s^{-1}})=\ell(A|R)|F|.$$
It follows that $\ml(A\delta_e|R^\Gamma)=\ell(A|R).$
By the upper continuity of $\ell$, we conclude
$$\ml(R\Gamma|R^\Gamma)=\ml(R\delta_e|R^\Gamma)=\sup_{A \in \cF^0(R)} \ml(A\delta_e|R^\Gamma)=\sup_{A \in \cF^0(R)}\ell(A|R)=\ell(R|R).$$
\end{proof}

In \cite[Corollary 6.5]{G10}, it was computed that $h(K^\bZ)=+\infty$ for any nontrivial abelian group $K$.   Using a general embedding construction due to Hanfeng Li (included with his permission), we can generalize this result to the following setting.
\begin{proposition} \label{product infinity}
Suppose that $\Gamma$ is an infinite countable amenable group and $\ell(R|R)> 0$. If $\ell$ is independent of ambient modules or $R$ is injective as  an $R$-module, we have $\ml(R^\Gamma)=+\infty$.
\end{proposition}

\begin{proof}
    Assume that $\varphi_n: (R\Gamma)^n \to R^\Gamma$ is an embedding as $R\Gamma$-modules for each $n \in \bN$. If $\ell$ is independent of ambient modules,  by Proposition \ref{tensor case}, as $n$ tends to the infinity, we have
    $$\ml(R^\Gamma) \geq \ml(\Ima\varphi|R^\Gamma)=\ml(\Ima \varphi_n|R^\Gamma)=\ml((\RG)^n|(\RG)^n) =n\ell(R|R)  \to +\infty.$$

    If $R$ is an injective $R$-module, then ${\rm Hom}_R(\cdot, R)$ is an exact functor. For any $\RG$-module $\cM$, there is a natural isomorphism from ${\rm Hom}_\RG(\cM, {\rm Hom}_R(\RG, R))$ to ${\rm Hom}_R(\cM, R)$ sending $\varphi$ to 
    $$\varphi'(x):=\varphi(x)(\delta_e)$$
    for any $x \in \cM$. It follows that $R^\Gamma \cong {\rm Hom}_R(\RG, R)$ is an injective $\RG$-module.  Applying the injectivity of $R^\Gamma$ to the inclusion map $i \colon (\RG)^n \to (R^\Gamma)^n$, there exists an $\RG$-module homomorphism $\psi_n \colon R^\Gamma \to (R^\Gamma)^n$ such that $\psi_n \circ\varphi_n=i$. By Corollary \ref{basic eg}, we have
    \begin{align*}
        \ml(R^\Gamma)&\geq \ml(\Ima \varphi_n|R^\Gamma)\\
        &\geq \ml(\Ima(\psi_n\circ \varphi_n)|(R^\Gamma)^n)\\
        &=\ml((\RG)^n|(R^\Gamma)^n)\\
        &=n\ml(\RG|R^\Gamma)=n\ell(R|R)  \to +\infty.
    \end{align*}
    
    Now we provide the desired embedding. This construction actually works for every infinite group $\Gamma$. For each $F \in \cF(\Gamma)$, each $s \in F$, and each $j \in \{1, 2,\dots, n\}$, choose $g_{F, s, j} \in \Gamma$ such that the sets $\{F^{-1}g_{F, s, j}: F \in \cF(\Gamma), s \in F, 1\leq j \leq n\}$ are pairwise disjoint. Since $\Gamma$ is infinite, one can easily find such elements $g_{F, s, j}$ satisfying this condition. For each $j \in \{1, 2, \dots, n\}$, define $b_j \in R^\Gamma$ to be $1$ at the positions $s^{-1}g_{F, s, j}$ for all  $F \in \cF(\Gamma)$  and all $s \in F$, and $0$ elsewhere.     It is then straightforward to verify that the map $(R\Gamma)^n \to R^\Gamma$ sending $(x_1,\dots, x_n)$ to $\sum_{j=1}^n x_jb_j$ is an injective $R\Gamma$-module homomorphism.
\end{proof}

\begin{example}
\begin{enumerate}
    \item Algebraic entropy $h$ can distinguish $(\bZ/2\bZ)\Gamma$ from $(\bZ/3\bZ)\Gamma$ as $\ZG$-modules, as Corollary \ref{basic eg} yields
    $$h((\bZ/2\bZ)\Gamma)=\log2\neq \log3= h((\bZ/3\bZ)\Gamma);$$
    However, it cannot distinguish $((\bZ/2\bZ)\Gamma)^2$ from $(\bZ/4\bZ)\Gamma$.
    \item Let $\cM=\ZG f$ for some nonzero $f=\sum_{s \in \Gamma} f_s s \in \ZG$. From \cite[Lemma 4.7]{CT15}, for any $N \in \bN$, one has
$$h([0, N)f|\cM) \geq \frac{\log N}{|K|^2},$$
where $K=\{s \in \Gamma: f_s \neq 0\}$ is the support of $f$. It follows that $h(\cM)=\infty$.

    \item Consider the mean weak length $\ml_k$ defined from the weak length $\ell_k$ as in Example \ref{wlexample} (3). For any $n \geq 2$, it can distinguish $((\bZ/k\bZ)\Gamma)^n$ from $(\bZ/k^n\bZ)\Gamma$ as $\ZG$-modules. To see this, using Propositions \ref{sum is additive} and \ref{generating prop}, we have
    $$\ml_k(((\bZ/k\bZ)\Gamma)^n)=n\ml_k((\bZ/k\bZ)\Gamma)=n\ml_k((\bZ/k\bZ)\delta_e) =n\log k.$$
    Meanwhile, using Proposition \ref{generating prop}, we have
    \begin{align*}
        \ml_k((\bZ/k^n\bZ)\Gamma)&=\ml_k((\bZ/k^n\bZ) \delta_e)\\
        &=\log|\{k^{n-1}i\mod k^n: i=0, 1, \cdots, k-1\}|\\
        &=\log k.
    \end{align*}
\end{enumerate}    
\end{example}

\section{Addition formula}

In this section, we present the proof of addition formula in Theorem \ref{main theorem}. The corollary \ref{main cor} then follows from Theorem \ref{main theorem} and Propositions \ref{length upgrading}, and \ref{log upgrading}. 

\begin{proof}[Proof of Theorem \ref{main theorem}]
Write $\pi \colon \cM_3 \to \cM_3/\cM_1$ for the quotient map. We first show the hard direction:
$$\ml(\cM_2|\cM_3) \leq \ml(\cM_1|\cM_3)+\ml(\pi(\cM_2)|\pi(\cM_3)).$$

Fix $\varepsilon >0$ and $A \in \cF^0(\cM_2)$. It suffices to find $B \in \cF^0(\cM_1)$ such that
$$\ml(A|\cM_3) \leq \ml(\pi(A)|\pi(\cM_3))+\ml(B|\cM_3) + 3\varepsilon+\varepsilon\ell(A|\cM_3).$$
For $\ml(\pi(A)|\pi(\cM_3))$ there exists $K_0 \in \cF(\Gamma)$ and $\delta_0 > 0$ such that for every $(K_0, \delta_0)$-invariant $F \in \cF(\Gamma)$, i.e. $|\{s \in F: K_0s \subseteq F\}|\geq (1-\delta_0)|F|$,  one has
\[
\ell(\pi(A)^{[F]}|\pi(\cM_3))
\leq\left(\ml(\pi(A)|\pi(\cM_3))+\varepsilon\right)|F|.
\]

By quasitiling Lemma \cite[Theorem 8.3]{L12}, there exist $ K_1 \in \cF(\Gamma)$ and $\delta_1>0$ and $ F_1,\dots, F_m \in \cF(\Gamma)$ satisfying that   
\begin{enumerate}
    \item[(i)] each $F_j$ is $(K_0,\delta_0)$-invariant and contains the identity of $\Gamma$;
    \item[(ii)] for every $(K_1, \delta_1)$-invariant $F \in \cF(\Gamma)$, there exist $D_1, \cdots, D_m \in \cF(\Gamma)$  such that 
    $$\bigsqcup_{j=1}^m \bigsqcup_{c \in D_j} F_jc \subseteq F$$
is a pairwise disjoint union and  $\left|F\setminus\bigsqcup_{j=1}^m F_jD_j\right|<\varepsilon|F|$.
\end{enumerate}

For each $j=1, \dots, m$, since $\ell(\cdot, \cdot|\cdot)$ is proper  (see Definition \ref{upgrading def}), there exists $B_j \in \cF(\cM_1)$ such that
$$\ell(A^{[F_j]}, B_j|\cM_3) < \ell(\pi(A^{[F_j]})|\pi(\cM_3)) +\varepsilon$$
(for algebraic entropy case $B_j$ can be chosen as$(A^{[F_j]}-A^{[F_j]})\cap \cM_1$).
Put $B=\cup_{j=1}^m B_j\cup \{0\}$. For $\ml(B|\cM_3)$, there exists $K_2 \in \cF(\Gamma)$ and $\delta_2 > 0$ such that for every $(K_2, \delta_2)$-invariant $F \in \cF(\Gamma)$, one has
\[
\ell(B^{[F]}|\cM_3)\leq \left(\ml(B|\cM_3)+\varepsilon\right)|F|.
\]

\textbf{Claim:} For any $(K_1, \delta_1)$-invariant $F \in \cF(\Gamma)$, we have
$$\ell(A^{[F]}, B^{[F]}|\cM_3)\leq |F|(\ml(\pi(A)|\pi(\cM_3))+2\varepsilon+\varepsilon \ell(A|\cM_3)).$$

Therefore, since $\ell(\cdot, \cdot)$ is proper, as $F \in \cF(\Gamma)$ is both $(K_1, \delta_1)$-invariant and $(K_2, \delta_2)$-invariant, we have
\begin{align*}
  \ell(A^{[F]}|\cM_3)& \leq\ell(A^{[F]}, B^{[F]}|\cM_3)+\ell(B^{[F]}|\cM_3)  \\
  &\leq |F|\bigl(\ml(\pi (A)|\pi(\cM_3))+2\varepsilon+\varepsilon\ell(A|\cM_3)\bigr)+|F|\bigl(\ml(B|\cM_3)+\varepsilon\bigr).
\end{align*}
Dividing both sides by $|F|$, and let $F$ gets more and more invariant, we obtain
\[
\ml(A|\cM_3)\leq \ml(\pi (A)|\pi(\cM_3))+\ml(B|\cM_3)+3\varepsilon+\varepsilon\ell(A|\cM_3).
\]

\noindent \textbf{Proof of Claim:}
Write $\varphi(F):=\ell(A^{[F]}, B^{[F]}|\cM_3)$.
Note that each $F_j$ is $(K_0, \delta_0)$-invariant, $0\in B_j$, and $e_{\Gamma}\in F_j$. Using the quasitiling $\{F_jD_j\}_j$  as above and Proposition \ref{basic bivariant} (1), we have

\begin{align*}
\varphi(F)&\leq\varphi\left(\bigsqcup_{j=1}^m F_jD_j\right)+\varphi(F\setminus \bigsqcup_{j=1}^m F_jD_j)\\
 &\leq\sum_{j=1}^{m}\varphi(F_j)|D_j|+\varepsilon|F|\ell(A|\cM_3) \\
 &\leq\sum_{j=1}^m|D_j|\ell(A^{[F_j]}, B_j^{[F_j]}|\cM_3)+\varepsilon|F|\ell(A|\cM_3) \\
&\leq\sum_{j=1}^m|D_j|\ell(A^{[F_j]}, B_j|\cM_3)+\varepsilon|F|\ell(A|\cM_3) \\
&\leq \sum_{j=1}^m|D_j|(\ell(\pi(A)^{[F_j]}|\pi(\cM_3))+\varepsilon)+\varepsilon|F|\ell(A|\cM_3) \\
&\leq\bigl(\ml(\pi(A)|\pi(\cM_3))+2\varepsilon\bigr)\sum_{j=1}^m|F_j||D_j|+\varepsilon |F|\ell(A|\cM_3)\\
&\leq\bigl(\ml(\pi(A)|\pi(\cM_3))+2\varepsilon\bigr)|F|+\varepsilon|F|\ell(A|\cM_3).
\end{align*}

Now we show the easy direction:
$$\ml(\cM_2|\cM_3) \geq \ml(\cM_1|\cM_3)+\ml(\pi(\cM_2)|\pi(\cM_3)).$$
For any $B \in \cF^0({\cM_1})$ and $C \in \cF^0(\cM_2/\cM_1)$, there exists $B_1 \in \cF^0(\cM_2)$ such that $C=\pi(B_1)$. Put $A =B+B_1$. Since $\ell$ has the strong quotient property (see Definition \ref{weak length def}), for any $F \in \cF(\Gamma)$, we have $A^{[F]}=B^{[F]}+ B_1^{[F]}$ and hence
$$\ell(A^{[F]}|\cM_3) = \ell(B^{[F]}+ B_1^{[F]}|\cM_3) \geq \ell(B^{[F]}|\cM_3)+\ell(C^{[F]}|\pi(\cM_3)).$$
Dividing both sides by $|F|$, and let $F$ gets more and more invariant, we obtain
$$\ml(\cM_2|\cM_3)\geq \ml(A|\cM_3)\geq \ml(B|\cM_3)+\ml(C|\pi(\cM_3))$$
and the conclusion follows.
\end{proof}

\begin{example}
 From the addition formula of mean weak length and Proposition \ref{product infinity}, we conclude that 
    $$\ml(R^\Gamma/R\Gamma)=+\infty$$ 
provided that $R$ as an $R$-module is injective or $\ell$ is independent of ambient modules. and$\ell(R|R) \in (0, +\infty)$.
\end{example}

We end our discussion with the following question.

\begin{question}
Does the mean weak length induced from Malcev rank in Example \ref{wlexample}
satisfy the addition formula?
\end{question}


\end{document}